\documentclass[11pt,english]{amsart}
\usepackage{amsmath,amssymb,amsfonts,amstext,amsthm}

\usepackage{babel}
\usepackage{latexsym}
\usepackage{ifthen}
\usepackage{xypic}

\usepackage[T1]{fontenc}
\usepackage{mathpazo}
\usepackage{hyperref}
\usepackage[paper = a4paper, 
left = 2.5cm, 
right = 2.5cm, 
top = 3cm, 
bottom = 2.95cm]{geometry}

\newcommand\sE{{\mathcal E}}

\newcommand\sF{{\mathcal F}}
\newcommand\sPS{{\mathcal PS}} 
\newcommand\sG{{\mathcal G}}

\newcommand\sI{{\mathcal I}}

\newcommand\sO{{\mathcal O}}
\newcommand\sS{{\mathcal S}}

\newcommand\sA{{\mathcal A}}

\newcommand\bQ{{\mathbb Q}}

\newcommand\bP{{\mathbb P}}

\newtheorem{theorem}{Theorem}[section]
\newtheorem{lemma}[theorem]{Lemma}
\newtheorem{proposition}[theorem]{Proposition}
\newtheorem{remark}[theorem]{Remark}
\newtheorem{corollary}[theorem]{Corollary}

\newtheorem{problem}[theorem]{Problem}

\newtheorem{conjecture}[theorem]{Conjecture}

\newtheorem{acknowledgement}[theorem]{Acknowledgement}

\theoremstyle{definition}
\newtheorem{example}[theorem]{Example}
\newtheorem{definition}[theorem]{Definition}

\begin{document} 

\title{Generically nef vector bundles and geometric applications}
\author{Thomas Peternell}

\address{Th. Peternell - 
Mathematisches Institut - Universit\"at Bayreuth - D-95440 Bayreuth, Germany} 
\email{thomas.peternell@uni-bayreuth.de}

\date{\today}
\maketitle
\tableofcontents


\begin{abstract}{The cotangent bundle of a non-uniruled projective manifold is generically nef, due to a theorem of Miyaoka. 
We show that the cotangent bundle is actually generically ample, if the manifold is of general type and study in
detail the case of intermediate Kodaira dimension. Moreover, manifolds with  generically nef and ample tangent bundles
are investigated as well as connections to classical theorems on vector fields on projective manifolds.}
\end{abstract} 

\section{Introduction} \label{Intro} 
Given a vector field $v$ on a complex projective manifold $X$, a classical theorem of Rosenlicht says that 
$X$ is uniruled, i.e., $X$ covered by rational curves, once $v$ has a zero. If on the other hand $v$ does not vanish at any point,
Lieberman has shown that there is a finite \'etale cover $\pi: \tilde X \to X$ and a splitting 
$$\tilde X \simeq  A \times Y$$
with an abelian variety $A$ such that the vector field $\pi^*(v)$ comes from a vector field on $A.$  In particular, if $X$ 
is of general type, then $X$ does not carry any non-zero vector field. \\
For various reasons it is interesting to ask 
what happens if $v$ is a section in $S^mT_X,$ or $(T_X)^{\otimes m},$ or even more general, in $(T_X)^{\otimes m} \otimes L$
with a numerically trivial line bundle $L$ on $X.$ In particular, one would like to have a vanishing 
\begin{equation} \label{eq0} H^0(X,(T_X)^{\otimes m} \otimes L) = 0 \end{equation} 
if $X$ is of general type and ask possibly for structure results in case $X$ is not uniruled. 
The question whether the above vanishing holds was communicated to me by N.Hitchin.
The philosohical reason for the vanishing is quite clear: if $X$ is of general type, then the cotangent bundle $\Omega^1_X$ 
should have some ampleness properties. One way to make this precise is to say that the restriction $\Omega^1_X \vert C$ is ample on 
sufficiently 
general curve $C \subset X.$ \\
There are two things to be mentioned immediately. First, a fundamental theorem of Miyaoka says that $\Omega^1_X \vert C$ 
is nef on the general curve; we say shortly that $\Omega^1_X $ is {\it generically nef}. Second, if $K_X$ is ample,
then $X$ admits a K\"ahler-Einstein metric, in particular $\Omega^1_X $ is stable and consequently $\Omega^1_X \vert C$ is
stable, from which it is easy to deduce that $\Omega^1_X \vert C$ is ample. 
\vskip .2cm We therefore ask under which conditions the cotangent bundle of a non-uniruled manifold is {\it generically ample. }
We show, based on \cite{BCHM09}, \cite{Ts88} and \cite{En88}, that generic ampleness indeed holds if $X$ is of general type, implying the vanishing \ref{eq0}.
We also give various results in case $X$ is not of general type, pointing to a generalization
of Lieberman's structure theorem. In fact, ``most'' non-uniruled varieties have generically ample cotangent bundles. Of course, if
$K_X$ is numerically trivial, then the cotangent bundle cannot be generically ample, and some vague sense, this should be the only
reason, i.e. if $\Omega^1_X$ is not generically ample, then in some sense $X$ should split off a variety with numerically
trivial canonical sheaf. However certain birational transformations must be allowed as well as \'etale cover. Also it is advisable
to deal with singular spaces as they occur in the minimal model program. One geometric reason for this picture is the fact that
a non-uniruled manifold $X$, whose cotangent bundle is not generically ample, carries in a natural way a foliation $\sF$ whose
determinant $\det \sF$ is numerically trivial (we assume that $K_X$ is not numerically trivial). If $\sF$ is chosen suitably, its leaves should 
be algebraic and lead to a decomposition of $X.$  
Taking determinants, we obtain
a section in $\bigwedge^q T_X \otimes L$ for some numerically trivial line bundle $L,$ giving the connection to the discussion we started with. 
\vskip .2cm 
The organization of the paper is as follows. We start with a short section on the movable cone, because the difference between the
movable cone and the ``complete intersection cone'' is very important in the framework of generic nefness. We also give an
example where the movable cone and the complete intersection cone differ (worked out with J.P.Demailly). In section 3 we discuss
in general the concept of generic nefness and its relation to stability. The following section is devoted to the study of 
generically ample cotangent bundles. In the last part we deal with generically nef tangent bundles and applications to 
manifolds with nef anticanonical bundles.

\section{The movable cone} 
\label{sec:1}

We fix a normal projective variety $X$ of dimension $n.$ Some notations first. Given ample line bundles
$H_1, \ldots, H_{n-1}, $ we set $h = (H_1, \ldots, H_{n-1})$ and simply say that $h$ is an ample class. 
We let 
$$ NS(X) = N^1(X) \subset  H^2(X,\mathbb R) $$
be the subspace generated by the classes of divisors and 
$$ N_1(X) \subset  H^{2n-2}(X,\mathbb R)  $$
be the subspace generated by the classes of curves.

\begin{definition}
\label{def1} \begin{enumerate} 
\item The {\it ample cone} $\sA$ is the open cone in $N^1(X)$ generated by the classes of ample line bundles, its closure is the 
{\it nef cone.}
\item The {\it pseudo-effective cone} $\sPS$ is the closed cone in $N^1(X)$ of classes of effective divisors. 
\item The {\it movable cone} $\overline{ME}(X) \subset N_1(X)$ 
is the closed cone generated by classes of curves of the form
$$  C = \mu_*(\tilde H_1 \cap \ldots \cap \tilde H_{n-1}) ; $$
here $\mu: \tilde X \to X$ is any modification from a pojective manifold $X$ and $\tilde H_i$ are very ample divisors in $\tilde X.$ 
These curves $C$ are called strongly movable. 
\item $\overline {NE}(X) \subset N_1(X) $ is the closed cone generated by the classes of irreducible curves. 
\item An irreducible curve $C$ is called movable, if $C = C_{t_0}$ is a member of a family $(C_t)$ of curves such that 
$X = \bigcup_t C_t.$ The closed cone generated by the classes of movable curves is denoted by $\overline {ME}(X).$
\item The complete intersection cone $\overline {CI}(X)$ is the closed cone generated by classes  $h = (H_1, \ldots, H_{n-1})$ with $H_i$ ample. 
 
\end{enumerate}
\end{definition}

Recall that a line bundle $L$ is {\it pseudo-effective} if $c_1(L) \in \sPS(X).$ The pseudo-effective line bundles are exactly those line bundles carrying a
singular hermitian metric with positive curvature current; see \cite{BDPP04} for further information.

\begin{example}
\label{ex1} {\rm We construct a smooth projective threefold $X$ with the property
$$ \overline{ME}(X) \ne \overline{CI}(X). $$ 
This example has been worked out in \cite{DP07}. 
We will do that by constructing on $X$ a line bundle which is on the boundary of the pseudo-effective cone, but strictly positive on 
$ \overline{CI}(X).$
\vskip .2cm \noindent
We choose two different points $p_1, p_2 \in \bP_2$ and consider a rank 2-vector bundle $E$ over $\bP_2$,  given as an extension 
\begin{equation} \label{eq1} 0 \to \sO \to E \to \sI_{\{p_1,p_2\}}(-2) \to 0  \end{equation}
(see e.g. [OSS80]). Observe $c_1(E) = -2;$ $c_2(E) = 2.$ Moreover, if $l \subset \bP_2$ is
the line through $p_1$ and $p_2$, then 
\begin{equation} \label{eq2} E \vert l = \sO(2) \oplus \sO(-4). \end{equation}
Set
$$ X = \bP(E)$$ with
tautological line bundle $$L = \sO_{\bP(E)}(1). $$ 
First we show that $L$ is strictly positive on $\overline{CI(X)}.$   In fact, fix the unique positive real number $a$ such that 
$$ L +  \pi^*(\sO(a))$$
is nef but not ample. Here  $\pi: X \to \bP_2$ is the projection. 
Notice that $a \geq 4$ by Equation \ref{eq2}. The nef cone of $X$ is easily seen to be 
generated by $\pi^*\sO(1)$ and $L+\pi^*\sO(a)$, hence
$ \overline{CI}(X)$ is a priori spanned by the three classes 
$(L + \pi^*(\sO(a))^2$, $\pi^*(\sO(1))^2$ and 
$\pi^*(\sO(1))\cdot(L + \pi^*(\sO(a))$. However
$$
L^2=c_1(E)\cdot L-c_2(E)=-2\pi^*\sO(1)\cdot L-2\pi^*\sO(1)^2.
$$
thus
$$(L + \pi^*(\sO(a))^2=
(2a-2)\pi^*\sO(1)\cdot L+(a^2-2)\pi^*\sO(1)^2,$$
and as $(a^2-2)/(2a-2)<a$ we see that 
$$\pi^*(\sO(1))\cdot(L + \pi^*(\sO(a))$$
is a positive linear combination of $(L + \pi^*(\sO(a))^2$ and
$\pi^*(\sO(1))^2$. Therefore the boundary of $ \overline{CI}(X)$ is spanned by 
$(L + \pi^*(\sO(a))^2$ and $\pi^*(\sO(1))^2\,$.
Now, using $a \geq 4$, we have
$$ L \cdot (L + \pi^*(\sO(a))^2=2-4a+a^2 \geq 2$$
and
$$L \cdot \pi^*(\sO(1))^2 = 1,$$
hence $L$ is strictly positive on $ \overline{CI}(X).$  
 
\vskip .2cm \noindent 
On the other hand, $L$ is effective since $E$ has a section, but it is clear 
from the exact sequence \ref{eq1}  that $L$ must be on the boundary of the 
pseudo-effective cone$\,$; otherwise $L - \pi^*(\sO(\epsilon)) $ would be 
effective (actually big) for small positive $\epsilon$). This is absurd 
because the tensor product of the exact sequence \ref{eq1} by $\sO(-\epsilon)$ realizes the
$\bQ$-vector bundle $E\otimes \sO(-\epsilon)$ as an extension of two 
strictly negative sheaves (take symmetric products to avoid
$\bQ$ coefficients$\,$!). Therefore $L$ cannot be strictly positive 
on $\overline {ME}(X)$.
}

\end{example} 

The fact that $\overline{ME}(X)$ and $\overline{CI}(X)$ disagree in general is very unpleasant and creates a lot of technical troubles.

It is a classical fact that the dual cone of $\overline{NE}(X)$ is the nef cone; the main result of \cite{BDPP04} determines the 
dual cone to the movable cone:

\begin{theorem} \label{dualitytheorem} 
The dual cone to $\overline{ME}(X)$ is the pseudo-effective cone $\sPS(X)$.  Moreover $\overline{ME}(X)$ is the closed cone 
generated by the classes of movable curves. 
\end{theorem}

It is not clear whether the dual cone to $\overline{CI}(X)$ has a nice description. Nevertheless we make the following

\begin{definition} \label{defgenericneflinebundles}
A line bundle $L$ is {\it generically nef} if $L \cdot h \geq 0$ for all ample classes $h.$
\end{definition}  

In the next section we extend this definition to vector bundles. Although generically nef line bundles are in general 
not pseudo-effective as seen in Example \ref{ex1}, this might still be true for the canonical bundle:

\begin{problem} Let $X$ be a projective manifold or a normal projective variety with (say) only canonical singularities. 
Suppose $K_X$ is generically nef. Is $K_X $ pseudo-effective?
\end{problem} 

In other words, suppose $K_X$ not pseudo-effective, which is the same as to say that $X$ is uniruled. Is there an ample class $h$
such that $K_X \cdot h < 0?$  This is open even in dimension 3; see \cite{CP98} for some results.

\section{Generically nef vector bundles}
\label{sec:2}

In this section we discuss generic nefness of general vector bundles and torsion free coherent sheaves.

\begin{definition} 
\label{basicdef} 
\begin{enumerate} 
\item Let $h = (H_1, \ldots, H_{n-1})$ be an ample class. 
A vector bundle $\sE$ is said to be $h-$  generically nef (ample), if $\sE \vert C$ is
nef (ample) for a general curve $C = D_1 \cap \ldots \cap D_{n-1}$ for general $D_i \in \vert m_i H_i \vert $ and $m_i \gg 0.$ 
Such a curve is called MR-general, which is to say ``general in the sense of Mehta-Ramanathan''. 
\item The vector bundle $\sE$ is called generically nef (ample), if $\sE$ is  $(H_1, \ldots, H_{n-1}) -$ generically nef (ample) for all $H_i.$
\item  $\sE$ is almost nef \cite{DPS01}, if there is a countable union $S$ of algebraic subvarieties such $\sE \vert C$ is nef for all curves
$C \not \subset S.$ 
\end{enumerate} 
\end{definition}

\begin{definition} Fix an ample class $h$ on a projective variety $X$ and let $\sE$ be a vector bundle on $X$. Then we define the slope
$$ \mu_h(\sE) = c_1(\sE) \cdot h  $$
and obtain the usual notion of (semi-)stability w.r.t. $h.$ 
\end{definition} 

The importance of the notion of MR-generality comes from Mehta-Ranamathan's theorem \cite{MR82}

\begin{theorem} \label{MR} Let $X$ be a projective manifold (or a normal projective variety) and $\sE$ a locally free sheaf on $X$.
Then $\sE$ is semi-stable w.r.t. $h$ if and only $\sE \vert C$ for $C$ MR-general w.r.t.  $h$.
\end{theorem} 

As a consequence one obtains 

\begin{corollary}  If $\sE$ is semi-stable w.r.t. $h$ and if 
$c_1(\sE) \cdot h \geq 0,$  then $\sE$ is generically nef w.r.t. $h$; in case of stability $\sE$ is even generically ample.
If $c_1(\sE) \cdot  h =  0,$
the converse also holds. 
\vskip .2cm \noindent
\end{corollary} 

The proof of Corollary 3.4 follows immediately from Miyaoka's characterization of semi-stable bundle on curves:

\begin{proposition} Let $C$ be a smooth compact curve and $\sE$ a vector bundle on $C$. Then $\sE$ is semi-stable if and only if the
$\mathbb Q-$bundle 
$$ \sE \otimes {\det \sE^* \over {r}} $$
is nef. 
\end{proposition} 

\begin{remark} Everything we said in this section remains true for coherent sheaves $\sE$ of positive rank $r$ which are locally free in 
codimension 1,
in particular for torsion free sheaves (the underlying variety being normal). \\
Recall that $\det \sE := (\bigwedge^r\sE)^{**}.$  
\end{remark} 

For later use we note the following obvious

\begin{lemma} \label{lemmasections} 
Let $X$ be an normal projective variety, $\sE$ a vector bundle or torsion free sheaf. 
\begin{enumerate}
\item If $\sE$ is $h-$generically ample for some $h,$ then $H^0(X,(\sE^*)^{\otimes m} \otimes L) = 0$ for all positive 
integers $m$ and all numerically trivial line bundles $L$ on $X.$ 
\item If $\sE$ is $h-$generically nef for some $h$ and $ 0 \ne s \in H^0(X,(\sE^*)^{\otimes m} \otimes L) = 0$
for some positive integer $m$ and some numerically trivial line bundle $L$, then $s$ does not have zeroes in codimension $1.$ 
\end{enumerate} 
\end{lemma} 

Nef bundles satisfy many Chern class inequalities.  Miyaoka \cite{Mi87} has shown that at least one also holds for generically nef bundles, once the determinant is nef:

\begin{theorem} \label{thm3} Let $X$ be an $n-$dimensional normal projective variety which is smooth in codimension 2. Let $\sE$ be a torsion free sheaf which is generically nef 
w.r.t. the polarization $(H_1, \ldots, H_{n-1}). $ If $\det \sE$ is $\mathbb Q-$Cartier and nef, 
then $$c_2(X) \cdot H_1 \cdot \ldots \cdot H_{n-2} \geq 0.$$
\end{theorem} 

This is not explicitly stated in \cite{Mi87}, but follows easily from ibid., Theorem 6.1. A Chern class inequality 
$$c_1^2(\sE) \cdot  H_1 \cdot \ldots \cdot H_{n-2} \geq c_2(\sE)  H_1 \cdot \ldots \cdot H_{n-2}$$
fails to be true: simply take a surface $X$ with $K_X$ ample and $c_1^2(X) < c_2(X)$ and let $\sE = \Omega^1_X$ (which is a generically nef
vector bundle, see the next section). Since generic nefness is a weak form of semi-stability, one might wonder wether there are Chern class inequalities
of type 
$$c_1(\sE)^2 \leq {{2r} \over {r-1}} c_2(\sE) \cdot h$$ 
(once $\det \sE$ is nef). In case $\sE = \Omega^1_X$, this is true, see again the next section.

\vskip .2cm 
If $\sE$ is a generically nef vector bundle, then in general there will through any given point many curves on which the bundle
is not nef. For an {\it almost nef} bundle (see Definition \ref{basicdef}), this will not be the case. Notice that in
case $\sE$ has rank $1,$ the notions ``almost nefness'' and ``pseudo-effectivity'' coincide. If a bundle is generically 
generated by its global sections, then $\sE$ is almost nef. Conversely, one has

\begin{theorem} Let $X$ be a projective manifold and $\sE$ a vector bundle on $X$. If $\sE$ is almost nef, then for any ample
line bundle $A$, there are positive numbers $m_0$ and $p_p0$ such that 
$$H^0(X, S^p((S^m\sE) \otimes A)) \ne 0$$
for $p \geq p_0$ and $m \geq m_0.$
\end{theorem} 

For the proof we refer to \cite{BDPP04}.
The question remains whether the bundles $ S^p((S^m\sE) \otimes A)$ can be even be generically generated. 
Here is a very special case, with a much stronger conclusion.

\begin{theorem} \label{almostnef} Let $X$ be an almost nef bundle of rank at most $3$ on a projective manifold $X$. 
If $ \det \sE \equiv 0,$ then
$\sE$ is numerically flat. 
\end{theorem}

A vector bundle $\sE$ is {\it numerically flat} if it admits a filtration by subbundles such that the graded pieces
are unitary flat vector bundles, \cite{DPS94}. For the proof we refer to \cite{BDPP04},7.6. The idea of the proof is as follows.
First notice that $\sE$ is semi-stable for all polarizations by Corollary \ref{CorMR}. This allows us to reduce to the case
that $\dim X = 2$ and that $\sE$ is stable for all polarizations. 
Now recall that
if $\sE$ is stable w.r.t. some polarization and if $c_1(\sE) = c_2(\sE) = 0,$
then $\sE$ is unitary flat, \cite{Ko87}. Hence it suffices to that $c_2(E) = 0.$ This is done by direct calculations
of intersection numbers on $\mathbb P(\sE).$ Of course there should be no reason why Theorem \ref{almostnef} should hold
only in low dimensions, but in higher dimensions the calculations get tedious. 

\begin{corollary} Let $X$ be a K3 surface or a Calabi-Yau threefold. Then $\Omega^1_X$ is not almost nef. 
\end{corollary} 

A standard Hilbert scheme arguments implies that there is a covering family $(C_t)$ for curves 
(with $C_t$ irreducible for general $t$), such that 
$\Omega^1_X \vert C_t$ is not nef for general $t.$

\section{The cotangent bundle}
\label{sec:3}

In this section we discuss positivity properties of the cotangent bundles of non-uniruled varieties. 
At the beginning there is Miyaoka's 

\begin{theorem} \label{miy} Let $X$ be projective manifold or more generally, a normal projective variety. If $X$ is not uniruled, then $\Omega^1_X$ is generically nef. 
\end{theorem} 

For a proof we refer to \cite{Mi87} and to \cite{SB92}. In \cite{CP07} this was generalized in the following form

\begin{theorem} \label{CP} Let $X$ be a projective manifold which is not uniruled. Let 
$$ \Omega^1_X \to Q \to 0 $$
be a torsion free quotient. Then $\det Q$ is pseudo-effective. 
\end{theorem}

Theorem \ref{CP} can be generalized to singular spaces as follows; the assumption on $\mathbb Q-$factoriality is needed in order 
to make sure
that $\det Q$ is $\mathbb Q-$Cartier (so $\mathbb Q-$factoriality could be substituted by simply assuming that the bidual of 
$\bigwedge^rQ$ is $\mathbb Q-$Cartier). 

\begin{corollary} \label{CP1} Let $X$ be a normal $\mathbb Q-$factorial variety. If $X$ is not uniruled, then the conclusion of 
Theorem \ref{CP} still holds.  
\end{corollary} 

\begin{proof}
Choose a desingularization $\pi: \hat X \to X$ and let 
$$ \Omega^1_X \to Q \to 0$$
be a torsion free quotient. We may assume that $\hat Q = \pi^*(Q) / {\rm torsion} $ is locally free. 
Via the canonical morphism $\pi^*(\Omega^1_X) \to \Omega^1_{\hat X},$ we obtain a rational map $ \Omega^1_{\hat X} \dasharrow  \hat Q$.
If $E$ denotes exceptional divisor with irreducible components $E_i,$ then this rational map yields a generically surjective map
$$ \Omega^1_{\hat X} \to \hat Q(kE) $$
for some non-negative imteger $k.$ Appyling Theorem \ref{CP}, $(\det \hat Q)(mE)$ contains an pseudo-effective divisor for some $m.$
Now 
$$ \det \hat Q = \pi^*(\det Q) + \sum a_i E_i, $$ 
with rational numbers $a_i$, hence $\det Q$ itself must be pseudo-effective (this can be easily seen in various ways). 
 
\end{proof}

\begin{corollary} \label{cor4} Let $f: X \to Y$ be fibration with $X$ and $Y$ normal $\mathbb Q-$Gorenstein. Suppose $X$ not uniruled. 
Then the relative 
canonical bundle $K_{X/Y}$ (which is $\mathbb Q-$Cartier) is pseudo-effective.
\end{corollary} 

A much more general theorem has been proved by Berndtsson and Paun \cite{BP07}. 

\vskip .2cm 
We consider a $\mathbb Q-$factorial normal projective variety which is not uniruled. The cotangent
sheaf $\Omega^1_X$ being generically nef, we ask how far it is from being generically ample. 

\begin{proposition} \label{flatquotient} 
Let $X$ be a $\mathbb Q-$factorial normal $n-$dimensional projective variety which is not uniruled. If $\Omega^1_X$ is not 
generically ample for some polarization $h,$ then there exists a torsion free quotient 
$$ \Omega^1_X \to Q \to 0 $$
of rank $1 \leq p \leq n$ such that $\det Q \equiv 0.$ \\
The case $p = n$ occurs exactly when $K_X \equiv 0.$ 
\end{proposition} 

\begin{proof}

Let $C$ be MR-general w.r.t $h.$ Let $\sS \subset  \Omega^1_X \vert C$ be the maximal ample subsheaf of the nef vector bundle
$\Omega^1_X \vert C,$ see \cite{PS00},2.3, \cite{PS04},p.636, \cite{KST07}, sect.6. Then the quotient $Q_C$ is numerically flat and $\sS_C$ is the maximal 
destabilizing subsheaf.
By 
\cite{MR82},$\sS_C$ extends to a reflexive subsheaf $\sS \subset \Omega^1_X$, which is $h-$maximal destabilizing. 
If $Q = \Omega^1_X /\sS$ is the quotient, then obviously $Q \vert C = Q_C.$ 
Now by Corollary \ref{CP1}, $\det Q$ is pseudo-effective. Since $c_1(Q) \cdot C = 0$, it follows that $\det Q \equiv 0.$ \\
Finally assume $p = n.$ Then $\Omega^1_X \vert C$ does not contain an ample subsheaf, hence $\Omega^1_X \vert C$ is numerically
flat; in particular $K_X \cdot  h = 0$. Since $K_X$ is pseudo-effective, we conclude $K_X \equiv 0.$ 

\end{proof}

So if $X$ is not uniruled and $\Omega^1_X$ not generically ample, then $K_X \equiv 0 $, or we have an exact sequence 
$$ 0 \to \sS \to \Omega^1_X \to Q \to 0 $$
with $Q$ torsion free of rank less than $n = \dim X$ and $\det Q \equiv 0.$ Dually we obtain 
$$ 0 \to \sF \to T_X \to T_X/\sF  \to 0$$
with $\det \sF \equiv 0.$ Since $(T_X/\sF )\vert C$ is negative in the setting of the proof of the last proposition, the natural morphism
$$ \bigwedge{^2}  \sF /{\rm torsion} \to T_X/\sF,$$
given by the Lie bracket, vanishes.Thus the subsheaf $\sF \subset T_X$ is a singular foliation, which we call a 
{\it numerically trivial
foliation.} So we may state

\begin{corollary} Let $X$ be $\mathbb Q-$factorial normal $n-$dimensional projective variety. Suppose $K _X \not \equiv 0.$ Then $\Omega^1_X$ is not generically
ample if and only if $X$ carries a numerically trivial foliation. 
\end{corollary} 

\vskip .2cm If $X$ is not uniruled, but $\Omega^1_X$ not generically ample, then we can take determinants in the setting of
Proposition \ref{flatquotient}, and obtain 

\begin{corollary} \label{cor6} Let $X$ be a $\mathbb Q-$factorial normal $n-$dimensional projective variety which is not uniruled. 
If $\Omega^1_X$ is not generically ample,  then there exists a $\mathbb Q-$Cartier divisor $D \equiv 0,$ a number $q$ and
a non-zero section in $H^0(X,(\bigwedge^qT_X)^{**} \otimes \sO_X(D)^{**}).$ In particular, if $X$ is smooth, then there is a numerically
flat line bundle $L$ such that $H^0(X,\bigwedge^qT_X \otimes L) \ne 0.$ 
\end{corollary}

\vskip .2cm 
Observe that the subsheaf $\sS \subset \Omega^1_X$ constructed in the proof of Proposition \ref{flatquotient} is $\alpha-$destabilizing 
for all
$\alpha \in \overline{ME} \setminus \{0\} $. Therefore we obtain

\begin{corollary} \label{corstab} Let $X$ be a $\mathbb Q-$factorial normal projective variety which is not uniruled.  If $\Omega^1_X$ is $\alpha-$semi-stable for some 
$\alpha \in \overline{ME} \setminus \{0\} $,
then $\Omega^1_X$ is generically ample unless $K_X \equiv 0.$ 
\end{corollary}

For various purposes which become clear immediately we need to consider certain singular varieties arising
from minimal model theory. We will not try to prove things in the greatest possible generality, but restrict to
the smallest class of singular varieties we need. We adopt the following notation.

\begin{definition} A terminal $n-$fold $X$ is a normal projective variety with at most terminal singularities which is also
$\bQ-$factorial. If additionally $K_X$ is nef, $X$ is called minimal. 
\end{definition}

Since the (co)tangent sheaf of a minimal variety $X$ is always $K_X-$semi-stable \cite{Ts88}, \cite{En88}, we obtain

\begin{corollary} \label{cor4a} 
Let $X$ be a minimal projective variety such that 
$K_X$ is big. 
Then $\Omega^1_X$ is generically ample. 
\end{corollary} 

Actually \cite{En88} gives more: $\Omega^1_X$ is generically ample for all smooth $X$ admitting a {\it holomorphic} map 
to a minimal variety. 
In general however a manifold of general type will not admit a holomorphic map to a minimal model. Nevertheless we can prove

\begin{theorem} \label{genample} 
Let $X$ be a projective manifold or terminal variety of general type. Then $\Omega^1_X $ is generically ample. 
\end{theorem}

\begin{proof} 

If $\Omega^1_X$ would not be generically ample, then we obtain a reflexive subsheaf $\sS \subset T_X$ such that $\det \sS \equiv 0.$ 
By \cite{BCHM09} there exists a sequence of contractions and flips
\begin{equation} f: X \dasharrow X' \end{equation} 
such that $X'$ is minimal. Since $f$ consists only of contractions and flips, we obtain an induced subsheaf $\sS' \subset T_{X'} $
such that $\det S' \equiv 0.$ Here it is important that no blow-up (``extraction'') is involved in $f.$ From Corollary \ref{cor4} we obtain 
a contradiction. 
\end{proof} 

Now Lemma \ref{lemmasections} gives 
\begin{corollary} \label{cor5} 
Let $X$ be a projective manifold of general type. Then $$H^0(X,(T_X)^{\otimes m }) = 0 $$ for all positive integers $m.$ 
\end{corollary}

More generally,  $H^0(X,(T_X)^{\otimes m } \otimes L^*) = 0 $ if $L$ is a pseudo-effective line bundle.  

\vskip .2cm
We now turn to the case that $X$ is not of general type. 
We start in dimension 2. 

\begin{theorem} Let $X$ be a smooth projective surface with $\kappa (X) \geq 0. $ 
Suppose that $H^0(X,T_X \otimes L) \ne 0$, where $L$ is a numerically trivial line bundle.
Then the non-trivial sections of $T_X \otimes L$ do not have any zeroes, in particular $c_2(X) = 0$ and one of the following holds up 
to finite \'etale cover.
\begin{enumerate}
\item $X$ is a torus
\item $\kappa (X) = 1$ and $X = B \times E$ with $g(B) \geq 2$ and $E$ elliptic.
\end{enumerate}
In particular, $X$ is minimal.\\
Conversely, if $X$ is (up to finite \'etale cover) a torus or of the form $X = B \times E$ with $g(B) \geq 2$ and $E$ elliptic,
then $H^0(X,T_X \otimes L) \ne 0$ for some numerically trivial line bundle $L.$

\end{theorem}

\begin{proof}
Fix a non vanishing section $s \in H^0(X,T_X \otimes L). $ Observe that due to Theorem \ref{miy} the section $s$ cannot have zeroes in codimension $1$. Thus $Z = \{s = 0\}$ is at most finite. 
Dualizing, we obtain an epimorphism
\begin{equation} 0 \to \sG \to  \Omega^1_X \to \sI_Z \otimes L^* \to 0 \end{equation}
with a line bundle $\sG \equiv K_X.$ 
From Bogomolov's theorem \cite{Bo79}, we have $\kappa (\sG) \leq 1,$ hence $\kappa (X) \leq 1.$ 
Next observe that if $L$ is torsion, i.e. $L^{\otimes m} = \mathcal O_X$ for some $m$, then after finite \'etale cover, we may
assume $L = \mathcal O_X;$ hence $X$ has a vector field $s$. This vector field cannot have a zero, otherwise $X$ would be uniruled
(see e.g. \cite{Li78}. Then a theorem of Lieberman \cite{Li78} applies and $X$ is (up to finite \'etale cover) a torus or a poduct $E \times C$ with $E$ 
elliptic and $g(C) \geq 2.$ \\
So we may assume that $L$ is not torsion; consequently $q(X) \geq 1.$ 

\vskip .2cm 
We first suppose that $X$ is minimal. If $\kappa (X) = 0$, then clearly $X$ is a torus up to finite \'etale cover. 
So let $\kappa (X) = 1.$ \\
We start by ruling out $g(B) = 0.$ In fact, if $B = \bP_1,$ then the semi-negativity of $R^1f_*(\mathcal O_X)$ together with $q(X) \geq 1$
shows via the Leray spectral sequence that $q(X) = 1.$ Let $g: X \to C$ be the Albanese map to an elliptic curve $C.$ Then (possibly
after finite \'etale cover of $X$), $L = g^*(L')$ with a numerically line bundle $L'$ on $C$, which is not torsion. Since the general
fiber $F$ of $f$ has an \'etale map to $C$, it follows that $L \vert F$ is not torsion. But then $H^0(F,T_X \otimes L \vert F) = 0,$ a
contradiction the existence of the section $s.$ Hence $g(B) \geq 1.$ \\
Consider the natural map
$$ \lambda: T_X \otimes L \to f^*(T_B) \otimes L. $$
Since $L$ is not torsion, $\lambda(s) = 0$ (this property of $L$ is of course only needed when $g(B) = 1).$ 
Therefore $s = \mu(s'),$ where 
\begin{equation} \mu: T_{X/B} \otimes L \to T_X \otimes L \end{equation}
is again the natural map. 
Recall that by definition $T_{X/B} = (\Omega^1_{X/B})^*,$ which is a reflexive sheaf of rank 1, hence a line bundle. 
Now recall that $s$ has zeroes at most in a finite set, so does $s'$. Consequently
$$ T_{X/B} \otimes L = \mathcal O_X. $$
On the other hand
$$  T_{X/B} = -K_X \otimes f^*(K_B) \otimes \mathcal O_X(\sum (m_i-1) F_i), $$
where the $F_i$ are the multiple fibers.
Putting things together, we obtain
$$ K_{X/B} = L \otimes  \mathcal O_X(\sum (m_i-1) F_i).$$
Since $K_{X/B}$ is pseudo-effective (see Corollary \ref{cor4} we cannot have any multiple fibers, hence $K_{X/B} \equiv 0.$ 
It follows that 
$f$ must be locally trivial (see e.g. \cite{BHPV04}, III.18, and also that $g(B) \geq 2.$ Then $X$ becomes actually a product after finite
\'etale cover.

\vskip .2cm 
We finally rule out the case that $X$ is not minimal. So suppose $X$ not minimal and let $\sigma: X \to X'$ be the blow-down of a $(-1)-$curve to a point $p.$
Then we can write $L =  \sigma^*(L')$ with some numerically trivial line bundle $L'$ on $X'$ and the section $s$ induces a section 
$s' \in H^0(X',T_{X'} \otimes L').$ Notice that $\sigma_*(T_X) = \sI_p \otimes T_{X'},$ hence $s'(p) = 0.$ 
Therefore we are reduced to the case where $X'$ is minimal and have to derive a contradiction. Now $s'$ 
has no zeroes by what we have proved before. This gives the contradiction we are looking for. 

\qed 
\end{proof} 

\begin{corollary} \label{cor10} Let $X$ be a smooth projective surface with $\kappa (X) \geq 0.$ The cotangent bundle $\Omega^1_X$ is not 
generically ample 
if and only if $X$ is a minimal surface with $\kappa = 0$ (i.e., a torus, hyperelliptic, K3 or Enriques) or $X$ is a minimal 
surface with $\kappa = 1$ and a locally trivial Iitaka fibration; in particular
$c_2(X) = 0$ and $X$ is a product after finite \'etale cover of the base.  
\end{corollary} 

We now turn to the case of threefolds $X$ - subject to the condition that $\Omega^1_X$ is not generically ample. 
By Theorem \ref{genample} $X$ is not of general type; thus we need only to consider the cases $\kappa (X) = 0,1,2.$
If $K_X \equiv 0$, then of course $\Omega^1_X$ cannot be generically ample. However it is still interesting to study 
numerically trivial foliations 
in this case. 

\begin{theorem} \label{threefolds} 
Let $X$ be a minimal projective threefold with $\kappa(X) = 0.$ 
Let 
$$ 0 \to \sF \to T_X \to Q \to 0$$
be a numerically trivial foliation, i.e., $\det \sF \equiv 0.$ 
Then there exists a finite cover $X' \to X$, \'etale in codimension 2, such that 
$X'$ is a torus or a product $A \times S$ with $A$ an elliptic curve and $S$ a K3-surface. 

\end{theorem} 

\begin{proof} 
By abundance, $mK_X = \sO_X$ for some positive integer $m$, since $X$ is minimal. 
By passing to a cover which is \'etale in codimension 2 and applying Proposition \ref{prop4} 
we may assume $K_X = \sO_X.$ 
We claim that 
$$ q(X) > 0,$$
possibly after finite cover \'etale in codimension 2. \\
If $\det Q$ is not torsion, then $q(X) > 0$ right away. 
If the $\bQ-$Cartier divisor $\det Q$ is torsion, then, after a finite cover \'etale in codimension 2, we obtain a holomorphic form of 
degree $1$ or $2.$ To be more precise, choose $m$ such that $m \det Q$ is Cartier. Then choose  $m'$ such that $m'm \det Q = \sO_X.$ 
Then there exists a finite cover $h: \tilde X \to X$, \'etale in codimension 2, such that the pull-back $h^*(\det Q)$ is trivial.
In the sheaf-theoretic language, $h^*(\det Q)^{**} = \sO_X.$ Now pull back the above exact sequence and 
conclude the existence
of a holomorphic 1-form in case $Q$ has rank 1 and a holomorphic 2-form in case $Q$ has rank 2. \\
Since $\chi(X,\sO_X) \leq 0$ by \cite{Mi87}, we conclude $q(X) \ne 0.$\\
Hence we have a non-trivial Albanese map
$$ \alpha: X \to {\rm Alb}(X) =: A.$$ 
By \cite{Ka85}, sect. 8, $\alpha$ is surjective with connected fibers. Moreover, possibly after a finite \'etale base change,
$X$ is birational to $F \times A$ where $F$ is a general fiber of $\alpha.$ \\
Suppose first that $\dim \alpha(X) = 1,$ i.e., $q(X) = 1.$ Then $F$ must be a K3 surface (after another finite \'etale cover).
Now $X$ is birational to $F \times A$ via a sequence of flops \cite{Ko89} and therefore $X$ itself is smooth (\cite{Ko89}, 4.11). Hence
by the Beauville-Bogomolov decomposition theorem, $X$ itself is a product (up to finite \'etale cover).   \\
The case $ \dim \alpha(X) = 2$ cannot occur, since then $X$ is birational to a product of an elliptic curve and a torus,
so that $q(X) = 3.$ \\
If finally $\dim \alpha(X) = 3, $ then $X$ is a torus.

\end{proof} 

In the situation of Theorem \ref{threefolds}, it is also easy to see that the foliation $\sF$ is induced by a foliation
$\sF'$ on $X'$ in a natural way. Moreover $\sF'$ is trivial sheaf in case $X'$ is a torus and it is given by the relative
tangent sheaf of a projection in case $X'$ is a product. 
\vskip .2cm 
From a variety $X$ whose cotangent bundle is not generically ample, one can construct new examples by the following devices. 

\begin{proposition} Let $f:X \dasharrow X' $ be a birational map of normal $\bQ-$factorial varieties which is an isomorphism
in codimension 1. Then $\Omega^1_X$ is generically ample if and only if $\Omega^1_{X'}$ is generically ample. 
\end{proposition} 

\begin{proof} 
 Suppose that $\Omega^1_X$ is generically ample and $\Omega^1_{X'}$ is not. Since $X'$ is not uniruled, $\Omega^1_{X'}$
is generically nef and by Proposition \ref{flatquotient} there is an exact sequence
$$ 0 \to \sS' \to \Omega^1_{X'} \to Q' \to 0 $$
such that $\det Q' \equiv 0.$ Since $f$ is an isomorphism in codimension 1, this sequence clearly induces a sequence
$$ 0 \to \sS \to \Omega^1_{X} \to Q \to 0 $$
such that $\det Q \equiv 0.$ 
Since the problem is symmetric in $X$ and $X'$, this ends the proof.
\end{proof}

\begin{proposition} \label{prop4} Let $f:X \to X'$ be a finite surjective map between normal projective $\bQ-$factorial varieties. Assume
that $f$ is \'etale in codimension 1.  Then $\Omega^1_X$ is generically ample if and only if $\Omega^1_{X'}$ is generically ample. 
\end{proposition} 

\begin{proof}
If $X'$ is not uniruled and $\Omega^1_{X'}$ is not generically ample, we lift a sequence
$$ 0 \to \sS' \to \Omega^1_{X'} \to Q' \to 0 $$ 
with $\det Q' \equiv 0$ and conclude that $\Omega_X^1$ is not generically ample. \\
Suppose now that $\Omega_X^1$ is not generically ample (and $X$ not uniruled). Then we obtain a sequence
$$ 0 \to \sS \to \Omega^1_{X} \to Q \to 0 $$
with $\det Q \equiv 0.$ 
If $\Omega^1_{X'} $ would be generically ample, then for a general complete intersection curve $C' \subset X'$
the bundle $\Omega^1_{X'} \vert C'$ is ample. Hence $\Omega_X^1 \vert f^{-1}(C') = f^*(\Omega^1_{X'} \vert C')$ is
ample, a contradiction. 
\end{proof} 

In view of the minimal model program we are reduced to consider birational morphisms which are ``divisorial'' in the sense that their exceptional locus contains a divisor. 
In one direction, the situation is neat:

\begin{proposition} Let $\pi: \hat X \to X$ be a birational map of normal $\mathbb Q-$factorial varieties. 
If $\Omega^1_X$ is generically ample, so does $\Omega^1_{\hat X}. $ 
\end{proposition}

\begin{proof}
If $\Omega_X^1$ would not be generically ample, we obtain an epimorphism
\begin{equation}  \Omega^1_{\hat X} \to \hat Q \to 0 \end{equation}
with a torsion free sheaf $\hat Q$ such that $\det \hat Q \equiv 0.$ 
Applying $\pi_*$ yields a map
$$ \mu: \pi_*(\Omega^1_{\hat X}) \to \pi_*(\hat Q), $$
which is an epimorphism in codimension 1. Since $\Omega^1_X = \pi_*(\Omega^1_{\hat X})$ outside a set of codimension at least 2, 
there exists a torsion free sheaf $Q$ coinciding with $\pi_*(\hat Q)$ outside a set of codimension at least 2 together
with an epimorphism 
$$ \Omega^1_X \to Q \to 0. $$
Since $\det Q = \det \pi_*(\hat Q) \equiv 0,$
the sheaf $\Omega^1_X$ cannot be generically ample. 
\qed
\end{proof} 

From a birational point of view, it remains to  investigate the following situation. Let $\pi: \hat X \to X$ be a divisorial 
contraction of non-uniruled terminal
varieties
and suppose that $\Omega^1_X$ is not generically ample. Under which conditions is $\Omega^1_{\hat X} $ generically ample?
Generic ampleness is not for free as shown in the following easy

\begin{example} Let $E$ be an elliptic curve and $S$ an abelian surface, say. Let $\hat S \to S$ be the blow-up at $p \in S$
and set $\hat X = E \times \hat S.$ Then $\hat X $ is the blow-up of $X = E \times S$ along the curve $E \times \{p\}$. 
Since $\Omega^1_{\hat X} = \sO_{\hat X} \oplus p_2^*(\Omega^1_{\hat S}),$ it cannot be generically ample
\end{example} 

We now study a special case of a point modification: the blow-up of a smooth point. 

\begin{proposition} Let $X$ be a non-uniruled $n-$dimensional projective manifold, $\pi: \hat X \to X$ the blow-up at the point $p.$ 
If $\Omega^1_{\hat X}$ is not generically ample, then there exists a number $q < n,$ a numerically trivial line bundle $L$ and
a non-zero section $v \in H^0(X,\bigwedge^qT_X \otimes L)$ vanishing at $p$: $v(p) = 0.$  
\end{proposition} 

\begin{proof} 
By Corollary \ref{cor6}, we get a non-zero section $\hat v \in H^0(\hat X, \bigwedge^qT_{\hat X} \otimes \hat L)$ for some numerically
trivial line bundle $\hat L.$ Notice that $\hat L = \pi^*(L)$ for some numerically trivial line bundle $L$ on $X$.  
Since 
$$ \pi_*(\bigwedge{^q}T_{\hat X}) \subset \bigwedge{^q}T_X,$$
we obtain a section $v \in H^0(X,\bigwedge^qT_X \otimes L).$
It remains to show that $v(p) = 0.$ This follows easily by taking $\pi_*$ of the exact sequence
$$ 0 \to \bigwedge{^q}T_{\hat X} \to \pi^*(\bigwedge{^q}T_X) \to \bigwedge{^q}(T_E(-1)) \to 0.$$
Here $E$ is the exceptional divisor of $\pi.$ In fact, taking $\pi_*$ gives 
$$ \pi_*(\bigwedge{^q}T_{\hat X}) = \sI_p \otimes T_X.$$

\end{proof}

From the Beauville-Bogomolov decomposition of projective manifolds $X$ with $c_1(X) = 0,$ we deduce immediately

\begin{corollary} Let $\hat X$ be the blow-up at a point $p$ in a projective manifold $X$ with $c_1(X) = 0$. Then
$\Omega^1_{\hat X}$ is generically ample. 
\end{corollary} 

Due to Conjecture \ref{c2} below this corollary should generalize to all non-uniruled manifolds $X.$ 
Based on the results presented here, one might formulate the following 

\begin{conjecture} \label{c1} {\it Let $X$ be a non-uniruled terminal $n-$fold. Suppose that $\Omega^1_X$ is not generically ample 
and $K_X \not \equiv 0$. Then, up to taking finite covers $X' \to X$, \'etale in codimension 1, and birational maps $X' \dasharrow X''$,which are 
biholomorphic in codimension 1, $X$ admits a locally trivial fibration, given by a
numerically trivial foliation, which is trivialized after another finite cover, \'etale in codimension 1.}
\end{conjecture} 

More generally, any numerical trivial foliation should yield the same conclusion. 

\vskip .2cm 
This might require a minimal model program, a study of minimal models in higher dimensions and possibly also a study of 
the divisorial Mori
contractions. In a subsequent paper we plan to study minimal threefolds $X$ with $\kappa (X) = 1,2$ whose cotangent bundles
is not generically ample and then study the 
transition from a general threefold to a minimal model.

\vskip .2cm 
We saw that a non-uniruled manifold $X$ whose cotangent bundle is not generically ample, admits a section $v$ in some bundle
$\bigwedge{^q} T_X \otimes L$, where $L$ is numerically trivial. It is very plausible that $v$
cannot have zeroes:

\begin{conjecture} \label{c2}{\it Let $X$ be a  projective manifold. Let $v \in H^0(X,\bigwedge^qT_X \otimes L)$ be a non-trivial section for some numerically trivial line bundle $L.$ If $v$ has a zero, then $X$ is uniruled. }
\end{conjecture} 

If $q = \dim X,$ then the assertion is clear by \cite{MM86}. If $ q = 1$ and $L$ is trivial, then the conjecture is a classical result, see e.g. \cite{Li78}. We will come back to Conjecture \ref{c2} at the end of the next section. 

\vskip .2cm 
A well-known, already mentioned theorem of Lieberman \cite{Li78} says that if a vector field
$v$ has no zeroes, then some finite \'etale cover $\tilde X$ of $X$ has the form $\tilde X = T \times Y$ with $T$ a torus, and $v$ comes from the torus. One might hope that this is simply a special case of a 
more general situation:

\begin{conjecture} \label{c3} {\it Let $X$ be a projective manifold, $L$ a numerically trivial line bundle and 
$$ v \in H^0(X,\bigwedge{^q} T_X \otimes L) $$
a non-zero section, where $q < \dim X.$ Then $X$ admits a finite \'etale cover $\tilde X \to X$ such that 
$\tilde X \simeq Y \times Z$ where $Y$ is a projective manifold with trivial canonical bundle and $v$ is induced by 
a section $v' \in H^0(Y,\bigwedge^{q} T_Y \otimes L').$}
\end{conjecture} 

\section{The tangent bundle}
\label{sec:4}

In this section we discuss the dual case: varieties whose tangent bundles are generically nef or generically ample. 
If $X$ is a projective manifold with generically nef tangent bundle $T_X$, then in particular $-K_X$ is generically nef. If $K_X$ is
pseudo-effective, then $K_X \equiv 0$ and the Bogomolov-Beauville decomposition applies. Therefore we will always assume that 
$K_X$ is not pseudo-effective, hence $X$ is uniruled. If moreover $T_X$ is generically ample w.r.t some polarization, then 
$X$ is rationally connected.  Actually much more holds:

\begin{theorem} Let $X$ be a projective manifold. Then $X$ is rationally connected if and only if there exists an
irreducible curve $C \subset X$ such that $T_X \vert C$ is ample.
\end{theorem} 

For the existence of $C$ if $X$ is rationally connected see \cite{Ko96}, IV.3.7; for the other direction we refer to
\cite{BM01}, \cite{KST07} and \cite{Pe06}.

The first class of varieties to consider are certainly Fano manifolds. One main problem here is the following standard

\begin{conjecture} {\it The tangent bundle of a Fano manifold $X$ is stable w.r.t. $-K_X$. }
\end{conjecture} 

This conjecture is known to be true in many cases, but open in general. 
Here is what is proved so far if $b_2(X) = 1.$ 

\begin{theorem} \label{stablefano} 
Let $X$ be a Fano manifold of dimension $n$ with $b_2(X) = 1.$ Under one of the following conditions the tangent bundle is stable. 
\begin{itemize} 
\item $n \leq 5$ (and semi-stable if $n \leq 6$);
\item $X$ has index $> {{n+1} \over {2}};$
\item $X$ is homogeneous;
\item $X$ (of dimension at least $3$ arises from a weighted projective space by performing the following operations: first take a smooth weighted complete
intersection, then take a cyclic cover, take again a smooth complete intersections; finally stop ad libitum.
\end{itemize} 
\end{theorem} 

For the first two assertions see \cite{Hw01}; the third is classical; the last is in \cite{PW95}. 

By Corollary 3.4, generic nefness, even generic ampleness, is a consequence of stability in case of Fano manifolds. 
Therefore generic nefness/ampleness is a
weak version of stability. So it is natural to ask for generic nefness/ampleness of the tangent bundle of Fano manifolds:

\begin{theorem} \label{thmgenample}  Let $X$ be a projective manifold with $-K_X$ big and nef. 
Then $T_X$ is generically ample (with respect to any polarization). 
\end{theorem} 

If $b_2(X) \geq 2,$ then of course the tangent bundle might not be (semi-)stable w.r.t. $-K_X;$ consider e.g. the
product of projective spaces (of different dimensions). \\ 
The proof of Theorem \ref{thmgenample} is given in \cite{Pe08}. The key to the proof is the following observation. 
Fix a polarization $h = (H_1, \ldots, H_{n-1}),$ where $n = \dim X.$ Suppose that $T_X$ is not $h-$generically ample. 
Since $-K_X \cdot h > 0,$ we may apply Corollary \label{CorMR} and therefore $T_X$ is not $h-$semi-stable More precisely, let $C$ be MR-general 
w.r.t. $h,$ then $T_X \vert C$ is not ample. Now we consider the Harder-Narasimhan filtration and find a piece $\sE_C$ which
is maximally ample, i.e., $\sE_C$ contains all ample subsheaves of $T_X \vert C.$ By the theory of Mehta-Ramanathan \cite{MR82},
the sheaf $\sE_C$ extends to a saturated subsheaf $\sE \subset T_X$. The maximal ampleness easily leads to the inequality
$$ (K_X + \det \sE) \cdot h > 0.$$
On the other hand, $K_X + \det \sE$ is a subsheaf of $\Omega^{n-k}_X.$ 
If $X$ is Fano with $b_2(X) = 1, $ then we conclude that $K_X + \det \sE$ must be ample, which is clearly impossible, e.g. by arguing
via rational
connectedness. In general we show, based on \cite{BCHM09}, that the movable $\overline{ME}(X)$ contains an extremal ray $R$ such that
$$ (K_X + \det \sE) \cdot R > 0.$$
This eventually leads, possible after passing to a suitable birational model, to a Fano fibration $f: X \to Y$ 
such that $K_X + \det \sE $ is relatively ample. This yields a contradiction in the same spirit as in the Fano case above.

\vskip .2cm With substantially more efforts, one can extend the last theorem in the following way.

\begin{theorem} \label{semi-ample} Let $X$ be a projective manifold with $-K_X$ semi-ample. Then $T_X$ is generically nef.
\end{theorem} 

From Theorem \ref{thm3} we therefore deduce

\begin{corollary} Let $X$ be an $n-$dimensional projective manifold with $-K_X$ semi-ample. Then 
$$c_2(X) \cdot H_1 \ldots \cdot H_{n-2} \geq 0 $$
for all ample line bundles $H_j$ on $X$.
\end{corollary} 

Of course Theorem \ref{semi-ample} should hold for all manifolds $X$ with $-K_X$ nef, and therefore also the inequality from the
last corollary should be true in this case. 

For biregular problems generic nefness is not enough; in fact, if $x \in X$ is a fixed point and $T_X$ is generically
nef, then it is not at all clear whether there is just one curve $C$ passing through $p$ such that $T_X \vert C$ is nef. 
Therefore we make the following

\begin{definition} Let $X$ be a projective manifold and $E$ a vector bundle on $X$. We say that $E$ is sufficiently nef
if for any $x \in X$ there is a family $(C_t)$ of curves through $x$ covering $X$ such that $E \vert C_t$ is
nef for general $t.$
\end{definition} 

We want to apply this to the study of manifolds $X$ with $-K_X$ nef:

\begin{conjecture}{\it Let $X$ be a projective manifold with $-K_X$ nef. Then the Albanese map is a surjective submersion.} 
\end{conjecture} 

Surjectivity is known by Qi Zhang \cite{Zh05} using char $p-$methods, smoothness of the Albanese map only in dimension at most 3
by \cite{PS98}.
The connection to the previous definition is given by 

\begin{proposition} Suppose that $T_X$ is sufficiently nef. Then the Albanese map is a surjective submersion. 
\end{proposition} 

\begin{proof} (cp. \cite{Pe08}). If the Albanese map would not be a surjective submersion, then there exists a holomorphic 
$1-$form $\omega$ on $X$ vanishing at some point $x.$ Now choose a general curve $C$ from a covering family
through $x$ such that $T_X \vert C$ is nef. Then $\omega \vert C$ is a {\it non-zero} section of $T_X^* \vert C$ having a zero.
This contradicts the nefness of $T_X \vert C.$  
\end{proof} 

Of course, a part of the last proposition works more generally:

\begin{proposition} \label{easy} If $E$ is sufficiently nef and if $E^*$ has a section $s$, then $s$ does not have any zeroes.
\end{proposition} 

We collect here some evidence that manifolds with nef anticanonical bundles have sufficiently nef tangent bundles and refer
to \cite{Pe08} for proofs.

\begin{theorem} Let $X$ be a projective manifold. 
\begin{itemize} 
\item If $E$ is a generically ample vector bundle, then $E$ is sufficiently ample.
\item If $-K_X$ is big and nef, then $T_X$ is sufficiently ample. 
\item If $-K_X$ is hermitian semi-positive, then $T_X$ is sufficiently nef. 
\end{itemize} 
\end{theorem} 

Notice however that a generically nef bundle need not be sufficiently nef; see \cite{Pe08} for an example
(a rank $2-$bundle on $\mathbb P_3$). 

\vskip .2cm We finally come back to Conjecture \ref{c2}. So suppose that $X$ is a projective manifold, let $L$ be numerically trivial
and consider a non-zero section
$$ v \in H^0(X,\bigwedge{^q}T_X \otimes L),$$
where $1 \leq q \leq \dim X-1.$
Applying Proposition \ref{easy}, Conjecture \ref{c2} is therefore a consequence of 

\begin{conjecture} \label{c4} Let $X$ be a non-uniruled projective manifold. Then $\Omega^1_X$ is sufficiently nef. 
\end{conjecture} 

Conjecture \ref{c4} is true in dimension 2 (using \cite{Pe08}, sect.7 and Corollary \ref{cor10}), and also if $K_X \equiv 0$ 
and if $\Omega^1_X$ is generically ample, again by  \cite{Pe08}, sect.7.

\begin{acknowledgement}
I would like to thank N. Hitchin for very interesting discussions during the inspiring conference in Hannover, 
which were the starting point of this paper. Thanks also go to F. Campana for discussions on the subject; actually many 
topics discussed here are based on our 
collaboration. Finally I acknowledge the support of the DFG Forschergruppe ``Classification of algebraic surfaces and compact complex manifolds''. 
\end{acknowledgement}

\end{document}